\documentclass[12pt]{smfart}

\usepackage{smfthm}
\usepackage{amssymb}
\usepackage[francais]{babel}

\newcommand{\Qp}{\mathbf{Q}_p}
\newcommand{\Fp}{\mathbf{F}_p}
\newcommand{\Zp}{\mathbf{Z}_p}
\newcommand{\Qq}{\mathbf{Q}_q}
\newcommand{\Fq}{\mathbf{F}_q}
\newcommand{\Zq}{\mathbf{Z}_q}
\newcommand{\ZZ}{\mathbf{Z}}
\newcommand{\QQ}{\mathbf{Q}}
\newcommand{\OO}{\mathcal{O}}
\newcommand{\MM}{\mathfrak{M}}
\newcommand{\Qpbar}{\overline{\mathbf{Q}}_p}

\newcommand{\Cp}{\mathbf{C}_p}
\newcommand{\val}{\operatorname{val}}
\renewcommand{\phi}{\varphi}

\renewcommand{\geq}{\geqslant}
\renewcommand{\leq}{\leqslant} 
\newcommand{\Gal}{\operatorname{Gal}}

\renewcommand{\tilde}{\widetilde}
\newcommand{\vp}{\val_p}
\newcommand{\ve}{\val_{\mathrm{E}}}
\newcommand{\eps}{\varepsilon}
\newcommand{\bcris}{\mathbf{B}_{\mathrm{cris}}} 
\newcommand{\be}{\mathbf{B}_{\mathrm{e}}} 
\newcommand{\bmax}{\mathbf{B}_{\mathrm{max}}} 
\newcommand{\brig}[2]{\mathbf{B}^{\dagger #1}_{\mathrm{rig} #2}}
\newcommand{\bdr}{\mathbf{B}_{\mathrm{dR}}}  
\newcommand{\atplus}{\widetilde{\mathbf{A}}^+}
\newcommand{\btplus}{\widetilde{\mathbf{B}}^+}
\newcommand{\et}{\widetilde{\mathbf{E}}}
\newcommand{\etplus}{\widetilde{\mathbf{E}}^+}
\newcommand{\dsen}{\mathrm{D}_{\mathrm{Sen}}}

\newcommand{\dcroc}[1]{[\![ #1 ]\!]}
\newcommand{\GL}{\mathrm{GL}}
\newcommand{\rhob}{\overline{\rho}}

\author{Laurent Berger}
\address{Universit\'e de Lyon \\
UMPA, ENS de Lyon \\
46 all\'ee d'Italie \\
69007 Lyon \\
France}
\email{laurent.berger@ens-lyon.fr}
\urladdr{www.umpa.ens-lyon.fr/\~{}lberger/}

\author{Ga\"etan Chenevier}
\address{CMLS \\
Ecole Polytechnique \\
91128 Palaiseau Cedex  \\
France}
\email{chenevier@math.polytechnique.fr}
\urladdr{www.math.polytechnique.fr/\~{}chenevier/}

\date{D\'ecembre 2009}

\title{Repr\'esentations potentiellement triangulines de dimension $2$}



\thanks{G.\ Chenevier est financ\'e par le C.N.R.S.}

\begin{document}

\begin{abstract}
Les deux r\'esultats principaux de cette note sont d'une part que si $V$ est une repr\'esentation de $G_{\Qp}$ de dimension $2$ qui est potentiellement trianguline, alors $V$ v\'erifie au moins une des propri\'et\'es suivantes (1) $V$ est trianguline d\'eploy\'ee (2) $V$ est une somme de caract\`eres ou une induite (3) $V$ est une repr\'esentation de de Rham tordue par un caract\`ere, et d'autre part qu'il existe des repr\'esentations de $G_{\Qp}$ de dimension $2$ qui ne sont pas potentiellement triangulines.
\end{abstract}

\begin{altabstract}
The two main results of this note are on the one hand that if $V$ is a $2$-dimensional potentially trianguline representation of $G_{\Qp}$ then $V$ satisfies at least one of the following properties (1) $V$ is split trianguline (2) $V$ is a direct sum of characters or an induced representation (3) $V$ is a twist of a de Rham representation, and on the other hand that there exists some $2$-dimensional representations of $G_{\Qp}$ which are not potentially trianguline.
\end{altabstract}

\maketitle

\tableofcontents

\setlength{\baselineskip}{18pt}

\section*{Introduction}

Dans le cadre de la correspondance de Langlands $p$-adique, Colmez a introduit dans \cite{CTR} la notion de repr\'esentation $p$-adique trianguline et a d\'emontr\'e plusieurs propri\'et\'es importantes de ces objets. Ses constructions ont \'et\'e reprises et g\'en\'eralis\'ees par Nakamura dans \cite{KN}. La d\'efinition de Colmez peut se faire en termes des {\og $(\phi,\Gamma)$-modules sur l'anneau de Robba \fg} de Fontaine et Kedlaya ou bien en termes de {\og $B$-paires \fg}. Dans cette note qui est un compl\'ement \`a \cite{CTR} et \cite{KN}, nous avons choisi de travailler avec les $B$-paires, qui sont plus commodes par certains aspects. A partir de maintenant, $K$ est une extension finie de $\Qp$ et $G_K = \Gal(\Qpbar/K)$. 

Les $B$-paires sont des objets introduits dans \cite{LB8} qui g\'en\'eralisent les repr\'esentations $p$-adiques, et la cat\'egorie des $B$-paires est alors \'equivalente \`a celle des $(\phi,\Gamma)$-modules sur l'anneau de Robba. Pour d\'efinir les $B$-paires, on utilise les anneaux $\bdr^+$, $\bdr$ et $\be=\bcris^{\phi=1}$ de Fontaine. Notons que ces trois anneaux sont filtr\'es et que leurs gradu\'es respectifs sont $\Cp[t]$, $\Cp[t,t^{-1}]$ et $\{ P \in \Cp[t^{-1}]$ tels que $P(0) \in \Qp\}$. 

Si $D$ est un $\phi$-module filtr\'e qui provient de la cohomologie d'un sch\'ema $X$ propre et lisse sur $\OO_K$, alors le $\phi$-module sous-jacent ne d\'epend que de la fibre sp\'eciale de $X$ (c'en est la cohomologie cristalline) alors que la filtration ne d\'epend que de la fibre g\'en\'erique (c'est la filtration de Hodge de la cohomologie de de Rham, dans laquelle se plonge la cohomologie cristalline). Si $V=V_{\mathrm{cris}}(D)$, alors on voit que $\be \otimes_{\Qp} V = (\bcris \otimes_{K_0} D)^{\phi=1}$ ne d\'epend que de la structure de $\phi$-module de $D$ et de plus, les $\phi$-modules $D_1$ et $D_2$ sont isomorphes si et seulement si les $\be$-repr\'esentations $\be \otimes_{\Qp} V_1$ et $\be \otimes_{\Qp} V_2$ le sont. Par ailleurs, $\bdr^+ \otimes_{\Qp} V = \mathrm{Fil}^0 (\bdr \otimes_{K_0} D)$ et les modules filtr\'es $K \otimes_{K_0} D_1$ et $K \otimes_{K_0} D_2$ sont isomorphes si et seulement si les $\bdr^+$-repr\'esentations  
$\bdr^+ \otimes_{\Qp} V_1$ et $\bdr^+ \otimes_{\Qp} V_2$ le sont. 

L'id\'ee sous-jacente \`a la construction des $B$-paires est d'isoler les ph\'enom\`enes li\'es \`a la {\og fibre sp\'eciale \fg} et \`a la {\og fibre g\'en\'erique \fg} en consid\'erant non pas des repr\'esentations $p$-adiques $V$, mais des {\og $B$-paires \fg} $W=(W_e,W_{dR}^+)$ o\`u $W_e$ est un $\be$-module libre muni d'une action semi-lin\'eaire et continue de $G_K$ et o\`u $W_{dR}^+$ est un $\bdr^+$-r\'eseau de $W_{dR} = \bdr \otimes_{\be} W_e$ stable par $G_K$. Si $V$ est une repr\'esentation $p$-adique, alors $W(V)= (\be \otimes_{\Qp} V, \bdr^+ \otimes_{\Qp} V)$ est une $B$-paire et cette construction permet de plonger la cat\'egorie des repr\'esentations $p$-adiques dans celle strictement plus grande des $B$-paires.

La d\'efinition de Colmez est alors la suivante : si $V$ est une repr\'esentation $E$-lin\'eaire de $G_K$ (ici $E$ est une extension finie de $\Qp$ qui est un corps de coefficients pour les repr\'esentations que l'on consid\`ere) alors on peut lui associer comme ci-dessus une $B$-paire $E$-lin\'eaire $W(V)$, et on dit que $V$ est trianguline d\'eploy\'ee si $W(V)$ est une extension successive de $B$-paires de rang $1$. On dit que $V$ est trianguline s'il existe une extension finie $F$ de $E$ telle que $F \otimes_E V$ est trianguline d\'eploy\'ee (notons cependant que les {\og triangulines \fg} de Colmez correspondent \`a nos {\og triangulines d\'eploy\'ees \fg}). On dit que $V$ est potentiellement trianguline s'il existe une extension finie $L$ de $K$ telle que $V \vert_{G_L}$ est trianguline. Le premier r\'esultat de cette note (le th\'eor\`eme \ref{main}) est le suivant.

\begin{enonce*}{Th\'eor\`eme A}
Si $V$ est une repr\'esentation $E$-lin\'eaire de $G_{\Qp}$ de dimension $2$ qui est potentiellement trianguline, alors $V$ v\'erifie au moins une des propri\'et\'es suivantes :
\begin{enumerate}
\item $V$ est trianguline d\'eploy\'ee;
\item $V$ est une somme de caract\`eres ou une induite;
\item $V$ est une repr\'esentation de de Rham tordue par un caract\`ere.
\end{enumerate}
\end{enonce*}

La d\'emonstration du th\'eor\`eme A se fait en utilisant de la descente galoisienne. Les $B$-paires ont des pentes (les pentes de frobenius des $(\phi,\Gamma)$-modules correspondants) et des poids (qui g\'en\'eralisent les poids de Hodge-Tate des repr\'esentations $p$-adiques). La combinatoire des pentes et des poids est assez rigide et si $V$ est une repr\'esentation potentiellement trianguline, soit sa triangulation descend et on est dans le cas (1), soit il y a des sym\'etries suppl\'ementaires suffisantes pour montrer qu'on est dans le cas (2) ou (3). 

Les conditions (1), (2) et (3) du th\'eor\`eme A ne sont pas du tout mutuellement exclusives, et en fait pour tout $S \subset \{ 1, 2, 3\}$ non vide, il existe une repr\'esentation $V$ qui v\'erifie exactement $S$. Le cas $S=\emptyset$ revient \`a la construction d'une repr\'esentation $p$-adique qui n'est pas potentiellement trianguline, et dans la suite de cet article, nous montrons que de telles repr\'esentations existent.

\begin{enonce*}{Th\'eor\`eme B}
Il existe des repr\'esentations $p$-adiques de dimension $2$
de $G_{\Qp}$ qui ne sont pas potentiellement triangulines.
\end{enonce*}

Nous prouvons en fait un r\'esultat plus fort, dont voici une cons\'equence.

\begin{enonce*}{Th\'eor\`eme C}
Soit $R : G_{\Qp} \to \GL_2(E)$ une repr\'esentation r\'esiduellement 
absolument irr\'eductible. Il existe une extension finie $F/E$ et une 
representation continue 
\[ \rho : G_{\Qp} \to \GL_2(F \langle t \rangle) \]
telle que $\rho_0 \otimes_E F=R$ et telle que pour tout 
$t \in \overline{\ZZ}_p$, sauf peut-\^etre pour un ensemble 
d\'enombrable d'entre eux, $\rho_t$ n'est pas potentiellement trianguline.
\end{enonce*}

Ce r\'esultat entra\^{\i}ne \'evidemment le pr\'ec\'edent : 
consid\'erer par exemple la repr\'esentation sur le module de 
Tate d'une courbe elliptique sur $\Qp$ ayant bonne r\'eduction 
supersinguli\`ere, ou encore une repr\'esentation induite convenable. 
La repr\'esentation $\rho$ que l'on construit est de polyn\^ome de Sen 
et d\'eterminant constants. Dans de nombreux cas, notamment dans 
l'exemple pr\'ec\'edent, on peut prendre $F=E$ dans l'\'enonc\'e ci-dessus.
  
Remarquons pour terminer que la d\'emonstration du th\'eor\`eme B n'est pas 
constructive, et le lecteur pourra chercher avec profit \`a construire 
explicitement une repr\'esentation qui n'est pas potentiellement trianguline. 

\section{La cat\'egorie des $B$-paires}
\label{defs}

Nous commen\c{c}ons par faire des rappels tr\`es succints sur les d\'efinitions (donn\'ees dans \cite{FPP} par exemple) des divers anneaux que nous utilisons dans cette note. Rappelons que $\etplus = \varprojlim_{x \mapsto x^p} \OO_{\Cp}$ est un anneau de caract\'eristique $p$, complet pour la valuation $\ve$ d\'efinie par $\ve(x)= \vp(x^{(0)})$ et qui contient un \'el\'ement $\eps$ tel que $\eps^{(n)}$ est une racine primitive $p^n$-i\`eme de l'unit\'e. On fixe un tel $\eps$ dans toute cette note. L'anneau $\et = \etplus[1/(\eps-1)]$ est alors un corps qui contient comme sous-corps dense la cl\^oture alg\'ebrique de $\Fp(\!(\eps-1)\!)$. On pose $\atplus=W(\etplus)$ et $\btplus=\atplus[1/p]$. L'application $\theta : \btplus \to \Cp$ qui \`a $\sum_{k \gg -\infty} p^k [x_k]$ associe $\sum_{k \gg -\infty} p^k x_k^{(0)}$ est un morphisme d'anneaux surjectif et $\bdr^+$ est le compl\'et\'e de $\btplus$ pour la topologie $\ker(\theta)$-adique, ce qui en fait un espace topologique de Fr\'echet. On pose $X=[\eps]-1 \in \atplus$ et $t=\log(1+X) \in \bdr^+$ et on d\'efinit $\bdr$ par $\bdr=\bdr^+[1/t]$. Soit $\tilde{p} \in \etplus$ un \'el\'ement tel que $\tilde{p}^{(0)}=p$. L'anneau $\bmax^+$ est l'ensemble des \'el\'ements de $\bdr^+$ qui peuvent s'\'ecrire sous la forme $\sum_{n \geq 0} b_n ([\tilde{p}]/p)^n$ o\`u $b_n \in \btplus$ et $b_n \to 0$ quand $n \to \infty$ ce qui en fait un sous-anneau de $\bdr^+$ muni en plus d'un frobenius $\phi$ qui est injectif, mais pas surjectif. On pose $\bmax=\bmax^+[1/t]$ et $\be=\bmax^{\phi=1}$. Rappelons que les anneaux $\bmax$ et $\bdr$ sont reli\'es, en plus de l'inclusion $\bmax \subset \bdr$, par la suite exacte fondamentale : $0 \to \Qp \to \be \to \bdr/\bdr^+ \to 0$. 

L'anneau $\bdr^+$ contient $\Qpbar$ ce qui fait que si $E$ est une extension finie de $\Qp$, alors $E \otimes \bdr^+ \simeq (\bdr^+)^{[E:\Qp]}$ (dans toute cette note, on \'ecrit $E\otimes -$ plut\^ot que $E \otimes_{\Qp} -$ pour all\'eger les notations). En ce qui concerne $E \otimes \be$, on a le r\'esultat suivant (rappelons qu'un anneau de B\'ezout est un anneau int\`egre tel que tout id\'eal de type fini est principal).

\begin{prop}\label{bebez}
Si $E$ est une extension finie de $\Qp$, alors l'anneau $E \otimes \be$ est un anneau de B\'ezout.
\end{prop}

\begin{proof}
Si $E=\Qp$, alors c'est la proposition 1.1.9 de \cite{LB8} et le cas g\'en\'eral fait l'objet du lemme 1.6 de \cite{KN}.
\end{proof}

Tous les anneaux construits ci-dessus admettent une action naturelle de $G_{\Qp}$ et donc de $G_K$ si $K \subset\Qpbar$. On fait agir $G_{\Qp}$ trivialement sur $E$ de sorte que $E \otimes \be$ et $E \otimes \bdr$ sont munis d'une action $E$-lin\'eaire de $G_{\Qp}$. 

Une $E \otimes \be$-repr\'esentation de $G_K$ est un $E \otimes \be$-module libre de rang fini muni d'une action semi-lin\'eaire de $G_K$. Si $W_e$ est une $E \otimes \be$-repr\'esentation de $G_K$, alors on pose $W_{dR} = (E \otimes \bdr) \otimes_{E \otimes \be} W_e$.

\begin{defi}\label{defbp}
Une $B^{\otimes E}_{\vert K}$-paire est une paire $W=(W_e,W_{dR}^+)$ o\`u $W_e$ est une $E \otimes \be$-repr\'esentation de $G_K$ et o\`u $W_{dR}^+$ est un $E \otimes \bdr^+$-r\'eseau de $W_{dR}$ stable par $G_K$.
\end{defi}

Si $E=\Qp$, alors on retrouve la d\'efinition du \S 2 de \cite{LB8} et dans le cas g\'en\'eral, on retrouve les {\og $E$-$B$-paires de $G_K$ \fg} de \cite{KN}. Rappelons que l'on note $X \subset W$ si $X_e \subset W_e$ et $X_{dR}^+ \subset W_{dR}^+$ mais que l'on dit que $X$ est un sous-objet strict de $W$ seulement si en plus $X_{dR}^+$ est satur\'e dans $W_{dR}^+$. Dans ce cas, le quotient $W/X$ est aussi une $B^{\otimes E}_{\vert K}$-paire.

Si $W$ est une $B^{\otimes E}_{\vert K}$-paire, et si $F$ est une extension finie galoisienne de $E$ et $L$ est une extension finie de $K$, alors $F \otimes_E W \vert_{G_L}$ est une $B^{\otimes F}_{\vert L}$-paire, munie en plus d'une action de $\Gal(F/E)$ et d'une action de $G_K$ qui \'etend celle de $G_L$. On a alors le r\'esultat suivant de {\og descente galoisienne \fg}.

\begin{prop}\label{galdesc}
Si $W$ est une $B^{\otimes E}_{\vert K}$-paire et si $X \subset F \otimes_E W \vert_{G_L}$ est  stable sous les actions de $\Gal(F/E)$ et de $G_K$, alors $X^{\Gal(F/E)}$ avec l'action induite de $G_K$ est une $B^{\otimes E}_{\vert K}$-paire et $X = F \otimes_E X^{\Gal(F/E)} \vert_{G_L}$.
\end{prop}

\begin{proof}
La proposition 2.2.1 de \cite{LB11} appliqu\'ee \`a $B=F$, $M=X$ et $S=E \otimes \be$ puis $S=E \otimes \bdr^+$ (si l'on remplace le produit tensoriel compl\'et\'e $\widehat{\otimes}$ par un produit tensoriel simple $\otimes$, ce qui ne change pas la d\'emonstration) implique que $X_e^{\Gal(F/E)}$ et $(X_{dR}^+)^{\Gal(F/E)}$ sont localement libres de rang fini sur $E \otimes \be$ et $E \otimes \bdr^+$ et v\'erifient $X = F \otimes_E X^{\Gal(F/E)}$. Ils sont libres de rang fini (par la proposition \ref{bebez} pour $X_e^{\Gal(F/E)}$ et car $\bdr^+$ est principal et le rang est constant pour $(X_{dR}^+)^{\Gal(F/E)}$) et comme $X^{\Gal(F/E)}$ est stable sous l'action induite de $G_K$, c'est bien une $B^{\otimes E}_{\vert K}$-paire.
\end{proof}

\section{Pentes et poids des $B$-paires}
\label{ppbp}

Rappelons que par le th\'eor\`eme A de \cite{LB8} (si $E=\Qp$) et par le th\'eor\`eme 1.36 de \cite{KN} en g\'en\'eral, on a une \'equivalence de cat\'egories entre la cat\'egorie des $B^{\otimes E}_{\vert K}$-paires et celle des $(\phi,\Gamma_K)$-modules sur l'anneau $E \otimes \brig{}{,K}$ o\`u $\brig{}{,K}$ est {\og l'anneau de Robba sur $K$ \fg}.  

On sait associer, par exemple selon la m\'ethode de \cite{KLMT}, des pentes aux $\phi$-modules sur l'anneau de Robba; en particulier, on dispose de la notion de $\phi$-module isocline de pente $s$ o\`u $s \in \QQ$ et on peut donc d\'efinir la notion de $B^{\otimes E}_{\vert K}$-paire isocline de pente $s$ via l'\'equivalence de cat\'egories entre $B^{\otimes E}_{\vert K}$-paires et $(\phi,\Gamma_K)$-modules sur l'anneau $E \otimes \brig{}{,K}$. On a alors le th\'eor\`eme suivant, qui suit de cette \'equivalence de cat\'egories et du th\'eor\`eme 6.10 de \cite{KLMT}, le th\'eor\`emeÊ de filtration par les pentes pour les $\phi$-modules sur l'anneau de Robba.

\begin{theo}\label{slopefil}
Si $W$ est une $B^{\otimes E}_{\vert K}$-paire, alors il existe une filtration canonique 
\[ 0 = W_0 \subset W_1 \subset \cdots \subset W_\ell = W \] 
par des sous-$B^{\otimes E}_{\vert K}$-paires, telle que :
\begin{enumerate}
\item pour tout $1 \leq i \leq \ell$, le quotient $W_i / W_{i-1}$ est isocline;
\item si l'on appelle $s_i$ la pente de $W_i / W_{i-1}$, alors $s_1 < s_2 < \cdots < s_\ell$.
\end{enumerate}
\end{theo}

Notons tout de m\^eme que $E \otimes \brig{}{,K}$ n'est pas n\'ecessairement int\`egre, et donc qu'il faut un petit argument suppl\'ementaire (on oublie $E$ puis on utilise le fait que la filtration est canonique pour le faire r\'eappara\^{\i}tre) pour obtenir le th\'eor\`eme ci-dessus, qui est alors le th\'eor\`eme 1.32 de \cite{KN}. Dans cette note, nous n'avons pas besoin de savoir comment calculer les pentes d'une $B$-paire. En plus du th\'eor\`eme \ref{slopefil}, nous n'utilisons que les deux r\'esultats ci-dessous.

\begin{prop}\label{etalrep}
Si $V$ est une repr\'esentation $E$-lin\'eaire de $G_K$ alors 
\[ W(V)=((E \otimes \be) \otimes_E V, (E \otimes \bdr^+) \otimes_E V) \] 
est une $B^{\otimes E}_{\vert K}$-paire, et le foncteur $V \mapsto W(V)$ donne une \'equivalence de cat\'egories entre la cat\'egorie des repr\'esentations $E$-lin\'eaires de $G_K$ et la cat\'egorie des $B^{\otimes E}_{\vert K}$-paires isoclines de pente nulle.
\end{prop}

\begin{proof}
Si $E=\Qp$, alors c'est le th\'eor\`eme 3.2.3 de \cite{LB8} appliqu\'e \`a $a/h=0/1$ et le cas $E$-lin\'eaire s'en d\'eduit imm\'ediatement.
\end{proof}

Si $W$ est une $B^{\otimes E}_{\vert K}$-paire de rang $1$, alors elle n'a qu'une seule pente, que l'on appelle aussi le degr\'e de $W$, not\'e $\deg(W)$. Si $W$ est de rang $\geq 1$, alors on pose $\deg(W) = \deg \det(W)$ ce qui fait de $\deg(\cdot)$ une fonction additive sur les suites exactes.

\begin{prop}\label{degrgun}
Si $X \subset W$ sont deux $B^{\otimes E}_{\vert K}$-paires de rang $1$, alors $\deg(X) \geq \deg(W)$ et $\deg(X)=\deg(W)$ si et seulement si $X=W$.
\end{prop}

\begin{proof}
Si $E=\Qp$, cela suit du corollaire 1.2.8 de \cite{LB8} et le cas $E$-lin\'eaire est tout \`a fait semblable.
\end{proof}

Les poids d'une $B$-paire $W$ sont une g\'en\'eralisation des poids de Hodge-Tate des repr\'esentations $p$-adiques. Rappelons que si $U$ est une $\Cp$-repr\'esentation de $G_K$, et que si l'on note $H_K = \Gal(\Qpbar/K(\mu_{p^\infty}))$ et $\Gamma_K = G_K / H_K$, alors la r\'eunion $U^{H_K}_{\mathrm{fini}}$ des sous-$K_\infty$-espaces vectoriels de dimension finie stables par $\Gamma_K$ de $U^{H_K}$ a la propri\'et\'e que l'application $\Cp \otimes_{K_\infty} U^{H_K}_{\mathrm{fini}} \to U$ est un isomorphisme (cf.\ \cite{SN80}). L'espace $U^{H_K}_{\mathrm{fini}}$ est muni de l'application $K_\infty$-lin\'eaire $\nabla_U = \log(\gamma) / \log_p(\chi(\gamma))$ avec $\gamma \in \Gamma_K \setminus \{1\}$ suffisamment proche de $1$. Le polyn\^ome caract\'eristique de $\nabla_U$ est alors \`a coeffcients dans $K$, et m\^eme dans $E \otimes K$ si $U$ est de plus $E$-lin\'eaire. Les racines de ce polyn\^ome sont les poids de Sen de $U$ (le nombre de racines d\'epend de la d\'ecomposition de $E \otimes K$, ce qui explique l'inclusion \'eventuellement stricte dans le (2) de la proposition \ref{incsen} ci-dessous). 

Si $W$ est une $B^{\otimes E}_{\vert K}$-paire, alors $W_{dR}^+ / t W_{dR}^+$ est une $E \otimes \Cp$-repr\'esentation de $G_K$ et on pose $\dsen(W) = (W_{dR}^+ / t W_{dR}^+)^{H_K}_{\mathrm{fini}}$. Les poids de $W$ sont alors les poids de Sen de $\nabla_W$ agissant sur $\dsen(W)$. Notons $\operatorname{poids}(W)$ l'ensemble des poids de Sen de $W$ compt\'es avec multiplicit\'e. La proposition ci-dessous est inspir\'ee des calculs du \S 3 de \cite{FIHP}. On \'ecrit ``$\operatorname{poids}(X) \subset \operatorname{poids}(W)+\ZZ_{\geq 0}$'' pour exprimer le fait que tout poids de $X$ est de la forme $w+a$ o\`u $w$ est un poids de $W$ et $a \in \ZZ_{\geq 0}$.

\begin{prop}\label{incsen}
Si $X \subset W$ sont deux $B^{\otimes E}_{\vert K}$-paires, alors 
\begin{enumerate}
\item $\operatorname{poids}(X) \subset \operatorname{poids}(W)+\ZZ_{\geq 0}$;
\item si $X$ est un sous-objet strict de $W$, alors $\operatorname{poids}(W) \supset \operatorname{poids}(W/X) \cup \operatorname{poids}(X)$;
\item si $X$ et $W$ sont de m\^eme rang et $\operatorname{poids}(X) = \operatorname{poids}(W)$, alors $X=W$.
\end{enumerate}
\end{prop}

\begin{proof}
Commen\c{c}ons par montrer le (2). Si $X$ est un sous-objet strict de $W$, alors on a une suite exacte $0 \to X_{dR}^+  \to W_{dR}^+  \to (W/X)_{dR}^+  \to 0$ et donc $0 \to \dsen(X) \to \dsen(W) \to \dsen(W/X) \to 0$ ce qui permet de conclure.

Le (2) implique que pour montrer le (1), on peut se ramener au cas o\`u $X$ et $W$ sont de m\^eme rang.  Notons $t W$ la $B$-paire $tW=(W_e,tW_{dR}^+)$ de sorte que $\operatorname{poids}(tW) = \operatorname{poids}(W) + 1$. Comme deux $\bdr^+$-r\'eseaux sont commensurables, il existe $h \geq 0$ tel que $t^h W \subset X$ et en consid\'erant la suite d'inclusions
\[ X = X + t^h W \subset X + t^{h-1} W \subset \cdots \subset X + t W \subset X + W = W, \]
on voit que pour montrer le (1), on peut en plus supposer que $tW \subset X \subset W$. Dans ce cas, on a deux suites exactes 
\begin{gather*} 
0 \to W_{dR}^+ / t W_{dR}^+ \to W_{dR}^+/X_{dR}^+   \\Ê0 \to t (W_{dR}^+ /  X_{dR}^+) \to X_{dR}^+ / t X_{dR}^+ \to  W_{dR}^+ / t W_{dR}^+
\end{gather*}
qui montrent que les poids de $X$ sont dans $\operatorname{poids}(W) \cup \operatorname{poids}(W)+1$.

Enfin pour montrer le (3), on voit que si $\operatorname{poids}(X) = \operatorname{poids}(W)$, alors on a \'egalit\'e \`a chaque \'etape ci-dessus et que cela implique $W=X$. On peut aussi se ramener au cas de rang $1$ en prenant le d\'eterminant.
\end{proof}

Si $W$ est une $B^{\otimes E}_{\vert K}$-paire, alors $W_{dR}$ est une $\bdr$-repr\'esentation de $G_K$ et ces objets sont \'etudi\'es dans le \S 3 de \cite{FIHP}. Si $W$ est une $B^{\otimes E}_{\vert K}$-paire, alors on dit qu'elle est de de Rham si la $\bdr$-repr\'esentation $W_{dR}$ est triviale. Remarquons que si $V$ est une repr\'esentation $E$-lin\'eaire de $G_K$ alors $V$ est de de Rham si et seulement si $W(V)$ est de de Rham.

\begin{prop}\label{isdr}
Si $W$ est une $B^{\otimes E}_{\vert K}$-paire \`a poids entiers et si $X \subset W$ est une $B^{\otimes E}_{\vert K}$-paire de rang $1$, alors $X$ est de de Rham.
\end{prop}

\begin{proof}
Comme $\Qpbar \subset \bdr$, on a une d\'ecomposition $E \otimes \bdr = \bdr^{[E:\Qp]}$ qui est compatible \`a l'action naturelle de $G_L$ sur les deux membres si $L$ est une extension de $\Qp$ qui contient $E$. Si l'on choisit une extension finie $L$ de $\Qp$ qui contient $K$ et $E$, alors on en d\'eduit une d\'ecomposition $(X \vert_{G_L})_{dR} = \oplus_{i=1}^{[E:\Qp]} X_{dR}^{(i)}$. Chaque $X_{dR}^{(i)}$ est une $\bdr$-repr\'esentation de dimension $1$ dont le poids appartient \`a $\ZZ$ et qui est donc triviale par les r\'esultats du \S 3.7 de \cite{FIHP}. C'est donc que $X \vert_{G_L}$ est de de Rham et donc $X$ aussi.
\end{proof}

\section{Repr\'esentations potentiellement triangulines}
\label{potrig}

Si $V$ est une repr\'esentation $E$-lin\'eaire de $G_K$, alors on peut lui associer par la proposition \ref{etalrep} une $B^{\otimes E}_{\vert K}$-paire $W(V)$, et on dit que $V$ est trianguline d\'eploy\'ee si $W(V)$ est une extension successive de $B^{\otimes E}_{\vert K}$-paires de rang $1$. On dit que $V$ est trianguline s'il existe une extension finie $F$ de $E$ telle que $F \otimes_E V$ est trianguline d\'eploy\'ee. Etant donn\'e l'\'equivalence de cat\'egories entre $B^{\otimes E}_{\vert K}$-paires et $(\phi,\Gamma_K)$-modules sur $E \otimes \brig{}{,K}$ cette d\'efinition est compatible avec les d\'efinitions 4.1 et 3.4 de \cite{CTR} (\`a ceci pr\`es que les {\og triangulines \fg} de Colmez correspondent \`a nos {\og triangulines d\'eploy\'ees \fg}). On dit que $V$ est potentiellement trianguline s'il existe une extension finie $L$ de $K$ telle que $V \vert_{G_L}$ est trianguline. 

\begin{theo}\label{main}
Si $V$ est une repr\'esentation $E$-lin\'eaire de $G_{\Qp}$ de dimension $2$ qui est potentiellement trianguline, alors $V$ v\'erifie au moins une des propri\'et\'es suivantes :
\begin{enumerate}
\item $V$ est trianguline d\'eploy\'ee;
\item $V$ est une somme de caract\`eres ou une induite;
\item $V$ est une repr\'esentation de de Rham tordue par un caract\`ere.
\end{enumerate}
\end{theo}

\begin{proof}
Soit $W=W(V)$ la $B^{\otimes E}_{\vert \Qp}$-paire isocline de pente nulle associ\'ee \`a $V$ comme dans la proposition \ref{etalrep}. Si $V$ est potentiellement trianguline, alors il existe une extension finie $F$ de $E$, et une extension finie $K$ de $\Qp$, que l'on peut supposer toutes les deux galoisiennes, telles que l'on puisse \'ecrire 
\[ 0 \to X \to F \otimes_E W \vert_{G_K} \to Y \to 0 \]
avec $X$ et $Y$ deux $B^{\otimes F}_{\vert K}$-paires de rang $1$. Si $g \in \Gal(F/E)$ ou bien si $g \in G_{\Qp}$, alors $g(X)$ et $g(Y)$ sont aussi deux $B^{\otimes F}_{\vert K}$-paires de rang $1$ et on a 
\[ 0 \to g(X) \to F \otimes_E W \vert_{G_K} \to g(Y) \to 0. \]
Si $g(X)=X$ quel que soit $g \in \Gal(F/E)$ et quel que soit $g \in G_{\Qp}$, alors la proposition \ref{galdesc} montre que $X$ provient d'une $B^{\otimes E}_{\vert \Qp}$-paire ce qui fait qu'on est dans le cas (1).

Le reste de la d\'emonstration est donc consacr\'e \`a examiner le cas o\`u il existerait un $g \in \Gal(F/E)$ ou bien un $g \in G_{\Qp}$ tel que $g(X) \neq X$. Dans ce cas, on a $g(X) \cap X = \{0\}$ et donc $g(X) \hookrightarrow Y$. Comme $W$ et donc $F \otimes_E W \vert_{G_K}$ est de pente nulle, le  th\'eor\`eme \ref{slopefil} implique que $\deg(X) \geq 0$. Si $\deg(X)=0$, alors $\deg (X \oplus g(X)) = 0$ et la proposition \ref{degrgun} appliqu\'ee \`a l'inclusion $\det(X \oplus g(X)) \subset \det(F \otimes_E W \vert_{G_K})$ montre que $F \otimes_E W \vert_{G_K}$ est somme directe de $X$ et $g(X)$. Par l'\'equivalence de cat\'egories de la proposition \ref{etalrep}, la repr\'esentation $F \otimes_E V \vert_{G_K}$ est somme directe de deux caract\`eres de $G_K$ ce qui fait par le lemme \ref{supersol} ci-dessous que l'on est dans le cas (2). 

Supposons donc que $\deg(X)>0$. Comme $V$ est une repr\'esentation $E$-lin\'eaire de $G_{\Qp}$ les poids de Sen de $V$ et donc de $W$ sont les deux racines $\lambda$ et $\mu \in \Qpbar$ du polyn\^ome caract\'eristique de $\nabla_W$ qui est \`a coefficients dans $E$. Les poids de $F \otimes_E W \vert_{G_K}$ sont donc des uplets de $\lambda$ et de $\mu$. Par le (2) de la proposition \ref{incsen}, les poids de Sen de $X$ et de $Y$ sont aussi des uplets de $\lambda$ et de $\mu$.

Supposons tout d'abord que $\lambda-\mu \notin \ZZ$. Toujours par le (2) de la proposition \ref{incsen}, le poids de $g(X)$ est aussi un uplet de $\lambda$ et de $\mu$ et comme $g(X) \hookrightarrow Y$ c'est par ailleurs un uplet d'\'el\'ements de $\lambda+\ZZ$ et de $\mu+\ZZ$ par le (1) de la proposition \ref{incsen}. Si $\lambda-\mu \notin \ZZ$, le poids de $g(X)$ est donc n\'ecessairement \'egal \`a celui de $Y$ et par le (3) de la proposition \ref{incsen}, on a $g(X)=Y$. Comme $\deg g(X) = \deg(X) > 0$ et $\deg(Y)<0$, on a une contradiction et le cas $\lambda-\mu \notin \ZZ$ ne peut pas se produire s'il existe $g$ tel que $g(X) \neq X$ avec $\deg(X)>0$.

Nous sommes donc ramen\'es \`a examiner la situation o\`u $\lambda-\mu = a \in \ZZ$ et $\deg(X)>0$. Quitte \`a tordre $V$ par un caract\`ere de poids $-\mu$, on peut supposer que les poids de Sen de $V$ sont $0$ et $a \in \ZZ$. Nous allons montrer que $V$ est alors de de Rham. Pour cela, remarquons que $X \oplus g(X) \subset F \otimes_E W \vert_{G_K}$ et que bien que cette inclusion ne soit pas une \'egalit\'e, on a $X_{dR} \oplus g(X_{dR}) = F \otimes_E W_{dR} \vert_{G_K}$ puisque les deux $\bdr$-espaces vectoriels sont de m\^eme dimension. Afin de terminer la d\'emonstration, on applique la proposition \ref{isdr} qui montre que $X$ et $g(X)$ sont de de Rham et donc $W$ aussi.
\end{proof}

\begin{lemm}\label{supersol}
Si $V$ est une repr\'esentation $E$-lin\'eaire de $G_{\Qp}$ de dimension $2$ telle qu'il existe une extension finie $K$ de $\Qp$ v\'erifiant $V \vert_{G_K} = V_1 \oplus V_2$ alors soit $V$ est une somme de caract\`eres, soit $V$ est une induite.
\end{lemm}

\begin{proof}
On peut supposer que $K$ est une extension galoisienne de $\Qp$. Si $V_1 \neq V_2$ alors le lemme ne pose pas de difficult\'e. Si $V_1 = V_2$ alors posons $H=G_K$ et $G=G_{\Qp}$. La th\'eorie de la ramification montre qu'il existe une suite de groupes 
\[ H = H_0 \subset H_1 \subset \cdots \subset H_n = G \]
telle que $H_i$ est distingu\'e dans $G$ et $H_{i+1}/H_i$ est cyclique (en d'autres termes, $G/H$ est hyper-r\'esoluble). Il existe alors $g_1,\hdots,g_n \in G$ tels que $H_i = \langle H_{i-1},g_i \rangle$. Si $V$ est une somme de caract\`eres, alors on a termin\'e et sinon il existe $1 \leq m \leq n$ tel que $V \vert_{H_{i-1}}$ est somme de deux caract\`eres \'egaux mais pas $V \vert_{H_i}$. La repr\'esentation $V \vert_{H_i}$ n'est pas irr\'eductible car $H_i = \langle H_{i-1},g_i \rangle$ et $H_{i-1}$ agit par des homoth\'eties, ce qui fait que quitte \`a remplacer $K$ par $\Qpbar^{H_i}$, on est ramen\'e au cas $V_1 \neq V_2$.
\end{proof}

\section{Parties fines d'un espace analytique}
\label{pfea}

Afin de montrer les th\'eor\`emes B et C de l'introduction, nous avons besoin de faire quelques 
rappels et compl\'ements sur les espaces rigides analytiques $p$-adiques. Si $\mathcal{X}$ est un tel 
espace et $x \in \mathcal{X}$ en est un point ferm\'e, nous d\'esignons par $K(x)=\OO_{\mathcal{X},x}/\MM_x$ le corps r\'esiduel de $x$, qui est une extension finie de $\Qp$. On rappelle que $\mathcal{X}$ est dit de dimension finie si l'ensemble $\{\dim \OO_{\mathcal{X},x}\}_{x \in \mathcal{X}}$ est born\'e, auquel cas $\dim(\mathcal{X})$ est le maximum de cet ensemble. Tous les espaces ci-dessous sont suppos\'es de dimension finie. Nous notons $\mathcal{B}$ la boule unit\'e ferm\'ee de dimension $1$ sur $\Qp$ (d'alg\`ebre affino\"ide $\Qp \langle t \rangle $).

Un espace rigide est dit \emph{de type d\'enombrable} s'il admet un recouvrement (non n\'eces\-sairement admissible) par un nombre d\'enombrable d'ouverts affino\"{\i}des. Cette propri\'et\'e est bien entendue stable par r\'eunions disjointes d\'enombrables quelconques.

\begin{lemm} 
\label{gcun}
Si $\mathcal{X}$ est un affino\"{\i}de, alors il existe une famille 
d\'enombrable d'ouverts affino\"{\i}des $\mathcal{U}=(\mathcal{U}_i)_{i \geq 0}$ 
de $\mathcal{X}$ telle que pour tout $x \in \mathcal{X}$ et tout voisinage 
ouvert affino\"{\i}de $\mathcal{V}$ de $x$, il existe un entier $i$ tel que $x \in \mathcal{U}_i \subset
\mathcal{V}$.
\end{lemm}

\begin{proof} 
Si $\mathcal{X}$ est la boule unit\'e $\mathcal{B}^n$, de param\`etres $t_1,\dots,t_n$, alors nous pouvons 
prendre pour $\mathcal{U}$ la collection de toutes les sous-boules affino\"{\i}des dont le centre $x$ est 
tel que les $t_i(x)$ sont alg\'ebriques sur $\QQ$ (rappelons que les rayons possibles sont d\'enombrables). 
En g\'en\'eral, nous pouvons trouver par d\'efinition une immersion ferm\'ee $\mathcal{X} \to \mathcal{B}^n$ 
pour $n$ assez grand, et l'image inverse dans $\mathcal{X}$ de la collection pr\'ec\'edente d'ouverts 
affino\"{\i}des de $\mathcal{B}^n$ a les propri\'et\'es requises.
\end{proof} 

\begin{coro} 
\label{gcde}
Si $\mathcal{X}$ est de type d\'enombrable, alors de tout recouvrement de $\mathcal{X}$ par des ouverts 
admissibles on peut extraire un recouvrement d\'enombrable. De plus, tout ferm\'e et tout ouvert de 
$\mathcal{X}$ est encore de type d\'enombrable. 
\end{coro}

\begin{proof} 
Pour le premier point, $\mathcal{X}$ \'etant de type d\'enombrable on peut le supposer affino\"{\i}de, 
auquel cas cela d\'ecoule du lemme pr\'ec\'edent. La seconde assertion est \'evidente pour un ferm\'e, 
et dans le cas d'un ouvert elle se ram\`ene \`a voir qu'un ouvert d'un affino\"{\i}de est de type 
d\'enombrable, ce qui d\'ecoule encore du lemme \ref{gcun} ci-dessus.
\end{proof}

\begin{defi} 
\label{gctr}
Si $\mathcal{X}$ est un espace rigide, une partie $\mathcal{A} \subset \mathcal{X}$ sera dite \emph{fine} 
s'il existe un espace rigide $\mathcal{Y}$ de type d\'enombrable, 
ainsi qu'un morphisme analytique $f : \mathcal{Y} \to \mathcal{X}$, tels que 
$\dim(\mathcal{Y}) < \dim(\mathcal{X})$ et $\mathcal{A} \subset f(\mathcal{Y})$.
\end{defi}

Par exemple, si $\dim(\mathcal{X})=1$, ses parties fines sont ses parties d\'enombrables. Le corps $\Qp$ \'etant 
ind\'enombrable il est bien connu qu'une telle partie est propre. Nous allons maintenant v\'erifier que ce r\'esultat s'\'etend en toute dimension.

\begin{lemm}
\label{gcqu} 
Soient $\mathcal{X}$ un espace analytique et $\mathcal{A} \subset \mathcal{X}$ une partie fine.
\begin{enumerate}
\item Si $\mathcal{U} \subset \mathcal{X}$ est un ouvert admissible de dimension $\dim(\mathcal{X})$, alors $\mathcal{A} \cap \mathcal{U}$ est une partie fine de $\mathcal{U}$.

\item Si $\nu :\mathcal{X}' \rightarrow \mathcal{X}$ est un morphisme fini tel que $\dim (\mathcal{X}') =\dim (\mathcal{X})$, alors $\nu^{-1}(A)$ est une partie fine de $\mathcal{X}'$.  Cela vaut en pariculier si $\nu$ est la normalisation de $\mathcal{X}$.

\item Si $\mathcal{X}$ est irr\'eductible et si $\{\mathcal{X}_\lambda \}_{\lambda \in \Lambda}$ est un ensemble ind\'enombrable de ferm\'es de $\mathcal{X}$ irr\'eductibles, deux \`a deux distincts, et de dimension $\dim(\mathcal{X})-1$, alors hors d'un ensemble d\'enombrable de $\lambda \in \Lambda$, la partie $\mathcal{A} \cap \mathcal{X}_\lambda$ est une partie fine de $\mathcal{X}_\lambda$.
\end{enumerate}
\end{lemm}

\begin{proof} 
Le (1) d\'ecoule de la d\'efinition et de ce qu'un ouvert d'un espace de type d\'enombrable l'est encore. Pour le (2), \'ecrivons $\mathcal{A} \subset f(\mathcal{Y})$ avec $\mathcal{Y}$ de type d\'enombrable et de dimension $< \dim (\mathcal{X})$ ; l'ensemble $\nu^{-1}(\mathcal{A})$ est inclus dans l'image du morphisme naturel $\mathcal{Y} \times_{\mathcal{X}}{\mathcal{X}}' \rightarrow \mathcal{X}'$. Cela permet de conclure car l'espace $\mathcal{Y} \times_{\mathcal{X}}\mathcal{X}'$ est de type d\'enombrable, \'etant  fini sur $\mathcal{Y}$ qui a cette propri\'et\'e, et de dimension $\leq \dim (\mathcal{Y} )< \dim (\mathcal{X}')$ pour la m\^eme raison.

V\'erifions \`a pr\'esent le (3). On peut supposer 
que $\mathcal{A}=f(\mathcal{Y})$ avec $\dim(\mathcal{Y})\leq n-1$ o\`u
$n=\dim (\mathcal{X})$. Soit $\Lambda' \subset \Lambda$ le sous-ensemble des $\lambda$ tels que 
$\mathcal{Y}$ ait une composante irr\'eductible $\mathcal{T}$ avec $f(\mathcal{T})$ Zariski-dense 
dans $\mathcal{X}_\lambda$. Comme $\mathcal{Y}$ est de type d\'enombrable, il en va de m\^eme de 
sa normalisation, de sorte que $\mathcal{Y}$ n'a qu'un nombre d\'enombrable de composantes irr\'eductibles, 
et donc $\Lambda'$ est d\'enombrable. 

Posons $\mathcal{A}_\lambda = \mathcal{A} \cap \mathcal{X}_\lambda$ et consid\'erons $\lambda \in \Lambda$ 
tel que $\mathcal{A}_\lambda$ n'est pas une partie fine de $\mathcal{X}_\lambda$. Nous allons montrer que 
$\lambda \in \Lambda'$. L'espace $\mathcal{Y}_\lambda=f^{-1}(\mathcal{X}_\lambda)$ est un ferm\'e de 
$\mathcal{Y}$ et en particulier il est de type d\'enombrable. Comme $\mathcal{A}_\lambda \subset f(\mathcal{Y}_\lambda)$ 
n'est pas une partie fine de $\mathcal{X}_\lambda$, l'espace $\mathcal{Y}_\lambda$ est de dimension $\geq n-1$. Il vient 
que $\dim(\mathcal{Y}_\lambda)= \dim (\mathcal{Y})=n-1$ car $\dim (\mathcal{Y})\leq n-1$. La d\'ecomposition en composantes irr\'eductibles de la nilr\'eduction de $\mathcal{Y}_\lambda$ est 
donc de la forme $\mathcal{T}\cup \mathcal{T}'$ o\`u $\mathcal{T}$ est une r\'eunion non vide de 
composantes irr\'eductibles $\mathcal{T}_i$ de $\mathcal{Y}$ et $\dim(\mathcal{T}') < n-1$. Si pour 
chaque $i$, l'adh\'erence Zariski $\mathcal{Z}_i$ de $f(\mathcal{T}_i)$ dans $\mathcal{X}_\lambda$ est 
stricte, donc de dimension $< \dim \mathcal{X}_\lambda$ par irr\'eductibilit\'e de $\mathcal{X}_\lambda$, alors 
$\mathcal{A}_\lambda$ est inclus dans la partie fine $f(\mathcal{T}') \cup (\cup_i \mathcal{Z}_i)$ de $\mathcal{X}_\lambda$. On en d\'eduit que l'un des $f(\mathcal{T}_i)$ est Zariski-dense dans $\mathcal{X}_\lambda$, 
et donc que $\lambda \in \Lambda'$.
\end{proof}

Si $K$ est une extension finie de $\Qp$, nous entendons par \emph{$K$-boule} de dimension $r$ l'affino\"{\i}de 
$\mathcal{B}^r_K$ sur $\Qp$ d'alg\`ebre $K \langle t_1,t_2,\hdots,t_r\rangle$. Si $\mathcal{X}$ est un affino\"{\i}de 
et si $x \in \mathcal{X}$ en est un point r\'egulier, rappelons qu'un r\'esultat classique d\^u \`a Kiehl 
\cite[Thm. 1.18]{K67} assure l'existence d'un voisinage ouvert affino\"{\i}de $\mathcal{U}$ de $x$ dans $\mathcal{X}$ 
qui est isomorphe \`a une $K(x)$-boule\footnote{Nous remercions Laurent Fargues de nous avoir indiqu\'e cette r\'ef\'erence.}.

\begin{prop}
\label{gcci} 
Si $\mathcal{X}$ est un espace analytique de dimension $>0$, alors une partie fine de $\mathcal{X}$ en est une partie stricte. 

Plus pr\'ecis\'ement, soit $\mathcal{A} \subset \mathcal{X}$ une partie fine et soit $x \in \mathcal{X}$ tel que 
$\dim \OO_{\mathcal{X},x}=\dim (\mathcal{X})>0$. Si $x$ est r\'egulier, alors on peut trouver un morphisme 
analytique $$\iota : \mathcal{B}^1_{K(x)} \longrightarrow \mathcal{X}$$ tel que $\iota(0)=x$, tel que $\iota^{-1}(\mathcal{A})$ 
est d\'enombrable, et qui est une immersion ferm\'ee vers un voisinage affino\"ide de $x$. Si $x$ n'est pas r\'egulier, 
on peut encore trouver un $\iota$ comme ci-dessus satisfaisant les deux premi\`ere conditions, si l'on s'autorise \`a 
remplacer $K(x)$ par une extension finie.
\end{prop}

\begin{proof} Soit $\mathcal{A}=f(\mathcal{Y}) \subset \mathcal{X}$ une partie fine et $x \in \mathcal{X}$ de dimension 
$\dim X >0$. D\'emontrons tout d'abord l'assertion concernant le cas o\`u $x$ est r\'egulier. D'apr\`es le (1) du lemme 
\ref{gcqu} et le r\'esultat de Kiehl, on peut supposer que $\mathcal{X}$ est une $K(x)$-boule de dimension $\dim (\mathcal{X})=n>0$. 
Si $n=1$ alors $\mathcal{A}$ est d\'enombrable et le r\'esultat est \'evident. Sinon on proc\`ede par r\'ecurrence sur $n$. On 
choisit une famille ind\'enombrable de sous-$K(x)$-boules ferm\'ees centr\'ees en $x$ et de dimension $n-1$ (par exemple $t_1=\lambda t_2$ pour $\lambda \in \mathbf{Z}_p^\times$), et on conclut par le (3) du lemme \ref{gcqu}.

La premi\`ere assertion de la proposition s'en d\'eduit car on peut supposer que $\mathcal{X}$ est r\'eduit, auquel cas son lieu r\'egulier est  un ouvert Zariski et Zariski-dense, donc contient un point de dimension $\dim \mathcal{X}$. 

V\'erifions le dernier point. Quitte \`a remplacer $\mathcal{X}$ par un ouvert affino\"{\i}de de dimension $\dim(\mathcal{X})$, et d'apr\`es le (1) du lemme \ref{gcqu}, on peut supposer que $\mathcal{X}$ est affino\"{\i}de contenant $x$, puis que $\mathcal{X}$ est normal et connexe d'apr\`es le (2) du m\^eme lemme. Par normalisation de Noether-Tate, on peut donc trouver un morphisme fini et surjectif $\pi: \mathcal{X}\rightarrow \mathcal{B}^n$ avec $n=\dim(\mathcal{X})$. Par cons\'equent, $\pi(\mathcal{A})$ est une partie fine de $\mathcal{B}^n$. Par l'argument pr\'ec\'edent, il existe une immersion ferm\'ee $\mathcal{B}^1_{K(\pi(x))} \rightarrow \mathcal{B}^n_{K(\pi(x))}$  ne rencontrant $\pi(\mathcal{A})$ qu'en un sous-ensemble d\'enombrable. L'image inverse $\mathcal{C}$ de cette boule dans $\mathcal{X}  \times K(\pi(x))$ est un ferm\'e d'\'equi-dimension $1$ contenant $x$ et ne recontrant $\mathcal{A}$ qu'en un sous-ensemble d\'enombrable. Quitte \`a normaliser $\mathcal{C}$, on peut finalement supposer que $\mathcal{X}$ est r\'egulier de dimension $1$, et on conclut encore par le r\'esultat de Kiehl.
\end{proof}

\section{Repr\'esentations non potentiellement triangulines}

\subsection{Rappel sur les espaces de d\'eformations} Soit $q$ une puissance de $p$, $\Fq$ le corps fini \`a $q$ \'el\'ements, 
$\Qq$ l'extension non ramifi\'ee de $\Qp$ de corps r\'esiduel $\Fq$ et $\Zq$ l'anneau des entiers de $\Qq$. Soit $$r : G_{\Qp} 
\to \GL_2(\Fq)$$ une repr\'esentation continue et absolument irr\'eductible. D'apr\`es un r\'esultat classique de Mazur (cf. 
\cite{MDG}), le foncteur des d\'eformations de $r$ aux $\Zq$-alg\`ebres locales finies de corps r\'esiduel $\Fq$ est 
pro-repr\'esentable par une $\Zq$-alg\`ebre locale noeth\'erienne compl\`ete $R(r)$ de corps r\'esiduel $\Fq$. On 
d\'esigne par  $\mathcal{X}(r)$ l'espace analytique $p$-adique associ\'e par Berthelot \`a $R(r)[1/p]$. D'apr\`es 
un th\'eor\`eme de Tate, si $\delta=\dim_{\Fq} {\rm Hom}_{G_{\Qp}}(r,r(1))$, alors $\dim_{\Fq} H^2(G_{\Qp},{\rm ad }(r))=\delta$ 
et $\dim_{\Fq}H^1(G_{\Qp},{\rm ad}(r))=5+\delta$. Lorsque $r \not \simeq r(1)$, ce qui est par exemple toujours satisfait si $p>3$, 
il vient que $R(r) \simeq \Zq\dcroc{t_0,t_1,\hdots,t_4}$, de sorte que $\mathcal{X}(r)$
est isomorphe \`a la boule unit\'e ouverte de dimension $5$ sur $\Qq$. Dans tous les cas, comme on le verra ci-dessous, $\mathcal{X}(r)$ est 
r\'egulier de dimension $5$.\footnote{Supposons $r \simeq r(1)$. Si $p=2$ ou $p=3$, on peut voir qu'en fait $\mathcal{X}(r)$ est 
la r\'eunion de $p$ boules unit\'es ouvertes sur $\Qq$ ; mieux, on a $R(r) \simeq
\mathbb{Z}_q\dcroc{t_0,t_1,\hdots,t_4}[\mu_p]$. En effet, c'est une observation du second auteur quand $p=2$, utilisant le morphisme naturel $\mu_2 \rightarrow G_{\mathbb{Q}_2}^{\rm ab} \rightarrow R(r)^\times$ ; le cas $p=3$ est plus subtil et a \'et\'e r\'ecemment obtenu par G. B\"ockle.}

Le $\Qq$-espace analytique $\mathcal{X}(r)$ jouit d'une propri\'et\'e universelle que nous rappelons \`a pr\'esent. Soit $\mathcal{Y}$ un $\Qq$-affino\"{\i}de et $\rho :  G_{\Qp} \to \GL_2(\OO(\mathcal{Y}))$ une repr\'esentation continue. Pour $y \in \mathcal{Y}$, on note $\rho_y : G_{\Qp} \to \GL_2(K(y))$ l'\'evaluation de $\rho$ en $y$. On note aussi $k_y$ le corps r\'esiduel de $K(y)$, qui est alors muni d'un morphisme naturel $\Fq \to k_y$, ainsi que $\rhob_y :   G_{\Qp} \to \GL_2(k_y)$ la repr\'esentation r\'esiduelle semi-simplifi\'ee de $\rho_y$. On dit que $\rho$ est {\it r\'esiduellement constante et \'egale \`a $r$} si $\rhob_y \simeq r \otimes_{\Fq} k_y$ pour tout $y \in \mathcal{Y}$. Si $\mathcal{Y}$ est connexe, il suffit pour cela que cela soit vrai pour un $y \in \mathcal{Y}$. {\it Les points de $\mathcal{X}(r)$ dans un $\Qq$-affino\"{\i}de $\mathcal{Y}$ sont en bijection canonique avec les classes d'isomorphisme de $\OO(\mathcal{Y})$-repr\'esentations continues $\rho : G_{\Qp} \to \GL_2(\OO(\mathcal{Y}))$ qui sont r\'esiduellement constantes et \'egales \`a $r$}. Cela vaut en particulier pour les points ferm\'es $x \in \mathcal{X}(r)$, qui sont en bijection avec les classes d'isomorphisme de rel\`evements $r_x : G_{\Qp} \to \GL_2(K(x))$ de $r$. Enfin, cette propri\'et\'e universelle appliqu\'ee aux $\Qq$-alg\`ebres locales artiniennes assure que pour tout $x$ dans $\mathcal{X}(r)$, $\widehat{\OO}_{\mathcal{X},x}$ est canoniquement isomorphe \`a la d\'eformation
(pro-)universelle de $r_x$ au sens de Mazur. Comme $r_x \not\simeq r_x(1)$ pout tout $x \in \mathcal{X}(r)$, 
les th\'eor\`emes de Tate montrent bien que $\mathcal{X}(r)$ est r\'egulier de dimension $5$.

\subsection{Points potentiellement triangulins de $\mathcal{X}(r)$} Etant donn\'ee une propri\'et\'e de repr\'esentations, on dira que $x \in \mathcal{X}(r)$ a cette propri\'et\'e si la repr\'esentation asscoi\'ee $r_x$ a cette propri\'et\'e.

\begin{conj}
\label{gcse}
L'ensemble des points potentiellement triangulins est une partie fine de $\mathcal{X}(r)$.
\end{conj}

Un point technique nous emp\^eche de d\'emontrer cette conjecture, mais nous en montrons ci-dessous une variante \`a poids de Hodge-Tate-Sen et d\'eterminant fix\'es. La th\'eorie de Tate-Sen nous fournit un polyn\^ome $P(T) = T^2+aT+b \in \OO(\mathcal{X}(r))[T]$ tel que pour tout $x \in \mathcal{X}(r)$ l'\'evaluation (des coefficients) de $P(T)$ en $x$ est le polyn\^ome de Sen de $r_x$. 
On dispose de plus d'une fonction $\lambda \in R(r)^\times$ qui est l'\'evaluation en $p \in \Qp^\times=G_{\Qp}^{\rm ab}$ 
du d\'eterminant de la repr\'esentation universelle $G_{\Qp} \to \GL_2(R(r))$. 

Fixons une extension finie $E$ de $\Qq$ ainsi que $P_0 \in E[T]$ unitaire de degr\'e $2$ et
$\lambda_0 \in \OO_E^\times$, et consid\'erons $\mathcal{X}_0 \subset \mathcal{X}(r) \times_{\Qq} E$ le ferm\'e 
d\'efini par les \'equations $P=P_0$ et $\lambda=\lambda_0$. C'est un espace rigide sur $E$ dont chaque composante irr\'eductible est de dimension $\geq 2$ par le \emph{hauptidealsatz} de Krull.\footnote{Un argument de torsion permet de voir que ces composantes irr\'eductibles sont de dimension $3$ au plus. Il est probable qu'elles soient toutes de dimension exactement $2$, mais ceci est inutile pour la suite.} 

\begin{theo} 
\label{gchu} Pour tout $(P_0,\lambda_0)$, l'ensemble des points potentiellement triangulins de chacune des composantes irr\'eductibles de $\mathcal{X}_0$ en est une partie fine.
\end{theo}

Ce th\'eor\`eme entra\^ine le Th\'eor\`eme $C$ de l'introduction par la proposition \ref{gcci}, ainsi donc que le th\'eor\`eme $B$. 

Soit $V$ une $E$-repr\'esentation trianguline de dimension $2$ qui est absolument
irr\'eductible et $D_{\rm rig}(V)$ le $(\phi,\Gamma)$-module \'etale sur l'anneau de Robba associ\'e \`a $V$. Soient $F$ une extension finie de $E$, ainsi que des
caract\`eres  $\delta_i : \Qp^\times \to F^\times$ pour $i=1,2$,  tels que
$D_{\rm rig}(V)\otimes_E F$ soit une
extension de $(\brig{}{,\Qp} \otimes F)(\delta_2)$ par $(\brig{}{,\Qp}
\otimes
F)(\delta_1)$. On note $x : \Qp^\times \rightarrow F^\times$ l'inclusion et on
pose $\chi=x|x|$ (c'est le caract\`ere cyclotomique). Rappelons que si $\delta_1\delta_2^{-1}\in \chi
x^{\mathbf{N}}$ alors quitte \`a tordre la repr\'esentation $V$ par un caract\`ere, soit $V$ est 
semistable non cristalline, soit $V$ est cristalline telle que le quotient des deux
valeurs propres de son frobenius cristallin est $p^{\pm 1}$. 

\begin{lemm} Soit $\OO$ la d\'eformation pro-universelle de $V$ aux
$E$-alg\`ebres artiniennes de corps r\'esiduel $E$, param\'etrant les d\'eformations de polyn\^ome de Sen et d\'eterminant
constants. Si $\delta_1\delta_2^{-1} \notin \chi
x^{\mathbf{N}}$, alors $\OO \simeq E\dcroc{X,Y}$.
\end{lemm}

\begin{proof} Si $\OO'$ est la d\'eformation pro-universelle de $V$ aux
$E$-alg\`ebres artiniennes de corps r\'esiduel $E$, alors on a 
d\'ej\`a vu que $\OO' \simeq E\dcroc{X_1,\dots,X_5}$. Soient $T^2+aT+b \in \OO[T]$ le polyn\^ome de Sen universel, $\lambda \in \OO^\times$ 
la valeur en $p$ du d\'eterminant de la d\'eformation universelle, et $(T^2+a_0T+b_0, \lambda_0)\in E[T] \times E^\times$ leurs 
\'evaluations en $0$. Par d\'efinition,
$$\OO=\OO'/(a-a_0,b-b_0,\lambda-\lambda_0).$$
D'apr\`es un r\'esultat classique sur les anneaux locaux r\'eguliers, il suffit de voir que les images 
de $a-a_0$, $b-b_0$ et $\lambda-\lambda_0$ sont lin\'eairement ind\'ependantes sur $E$ dans $\MM/\MM^2$ o\`u $\MM$ est l'id\'eal 
maximal de $\OO'$. Il suffit de le v\'erifier apr\`es extension des scalaires \`a
$F$, et donc de voir qu'il existe des d\'eformations $\widetilde{V}$ de $V
\otimes_E F$ \`a 
$F[\varepsilon]/(\varepsilon)^2$ telles que $(a(\widetilde{V})-a_0,b(\widetilde{V})-b_0,
\lambda(\widetilde{V})-\lambda_0)$ soit quelconque dans $(\varepsilon F)^3$.
Par l'hypoth\`ese sur $\delta_1\delta_2^{-1}$, cela r\'esulte de 
\cite[Prop. 2.3.10 (ii)]{BCH}, qui montre que l'on peut m\^eme choisir $\widetilde{V}$ trianguline sur 
$F[\varepsilon]/(\varepsilon^2)$ au sens de {\it loc.cit.}
\end{proof}

Ainsi, si la repr\'esentation $R$ de l'\'enonc\'e du th\'eor\`eme $C$
satisfait les hypoth\`eses du lemme ci-dessus, alors le point correspondant \`a $R$ dans
l'espace $\mathcal{X}_0 \subset \mathcal{X}(\overline{R})$ appropri\'e est
un point r\'egulier. Dans ce cas, on peut donc prendre $F=E$ dans l'\'enonc\'e
de ce th\'eor\`eme, d'apr\`es la proposition~\ref{gcci}. La pr\'ecision suivant l'\'enonc\'e du th\'eor\`eme $C$ de l'introduction 
s'en d\'eduit. Cela d\'emontre en particulier qu'il existe des
repr\'esentations non potentiellement triangulines qui sont \`a coefficients
dans $\Qp$.

\subsection{Preuve du th\'eor\`eme \ref{gchu}} Le reste du chapitre est consacr\'e \`a la d\'emonstration du th\'eor\`eme~\ref{gchu}. D'apr\`es le th\'eor\`eme A de l'introduction, il y a trois types (non exclusifs) de repr\'esentations potentiellement triangulines : 
\begin{itemize}
\item[(a)] les repr\'esentations de de Rham tordues par un caract\`ere;
\item[(b)] les induites d'un caract\`ere d'une extension quadratique de $\Qp$; 
\item[(c)] les repr\'esentations triangulines (d\'eploy\'ees).
\end{itemize}

Ces repr\'esentations vivent dans des familles analytiques naturelles que nous 
d\'ecrivons \`a pr\'esent. Consid\'erons tout d'abord le cas (b), qui est le plus 
simple. Il ne serait pas difficile d'expliciter la famille universelle (de
dimension $3$) form\'ee de toutes les repr\'esentations de type (b). C'est
cependant inutile pour l'application au th\'eor\`eme \ref{gchu} car \`a d\'eterminant et polyn\^ome de Sen
fix\'es, 
il n'y a qu'un nombre 
d\'enombrable de telles repr\'esentations. 

En effet, il n'y a d'une
part qu'un nombre fini d'extensions quadratiques
$K$ de $\Qp$ ($3$ si $p \geq 3$ et $7$ si $p=2$). D'autre part, fixons $K$ une extension finie de $\Qp$
et notons $\Sigma$ l'ensemble des $[K:\Qp]$ plongements de $K$ dans
$\Qpbar$. Si $\eta : G^{\rm ab}_K=\widehat{K^\times} \to
\Qpbar^\times$ est un caract\`ere continu, il existe des \'el\'ements uniques $a_\sigma \in
\Qpbar$, 
$\sigma \in \Sigma$, tels que  $\eta(x)=\prod_{\sigma \in
\Sigma}\sigma(x)^{a_\sigma}$ pour tout $x$ dans un sous-groupe ouvert assez
petit de $\OO_K^\times$. En particulier, le caract\`ere $\eta$ est
d'ordre fini si, et seulement si, $a_\sigma=0$ pour tout $\sigma \in
\Sigma$ ; la donn\'ee des $a_\sigma$ d\'etermine donc $\eta$ \`a
multiplication pr\`es par un caract\`ere d'ordre fini (et en particulier, un ensemble
d\'enombrable de caract\`eres). On conclut car le polyn\^ome de Sen
de ${\rm Ind}_{G_K}^{G_{\Qp}} \eta$ est exactement $\prod_{\sigma \in
\Sigma}(T-a_\sigma)$. 

Int\'eressons nous maintenant au cas (c). Colmez a d\'efini dans \cite{CTR} l'espace $\mathcal{S}$ des repr\'esen\-tations 
triangulines (nous nous limitons ici aux repr\'esentations irr\'eductibles). Par
cons\-truction, c'est un espace analytique 
sur $\Qp$ de type d\'enombrable, \'equi-dimensionnel de dimension $4$, et muni d'un
morphisme naturel vers l'espace 
$\mathcal{D}$ des caract\`eres $p$-adiques de $(\Qp^\times)^2$ (une
r\'eunion disjointe finie de copies de $\mathbb{G}_m^2 \times \mathcal{W}^2$
o\`u $\mathcal{W}$ est la boule unit\'e ouverte de dimension $1$ sur $\Qp$). L'espace d\'efini par Colmez est construit 
de mani\`ere ad-hoc de sorte que ses points ferm\'es param\`etrent les repr\'esentations triangulines. Contrairement \`a 
ce que l'on pourrait penser, il n'existe pas de famille analytique de repr\'esentations galoisiennes sur $\mathcal{S}$ qui 
se sp\'ecialise en tout point sur la repr\'esentation param\'etr\'ee par ce point ; une premi\`ere obstruction vient de ce 
que la repr\'esentation r\'esiduelle associ\'ee n'est pas constante sur $\mathcal{S}$.  Cependant, Colmez d\'emontre une 
forme faible de ce type d'\'enonc\'e qui est suffisante pour notre application (mais pas tout \`a fait pour la conjecture 
\ref{gcse}). Si $x \in \mathcal{S}$ est un point ferm\'e, il lui est associ\'e un point $\delta(x) \in \mathcal{D}$, c'est 
\`a dire une paire de caract\`eres continus $\delta_{i,x} : \Qp^\times \to K(x)^\times$ pour $i=1$, $2$.
Par construction, le $D_{\rm rig}$ de la $K(x)$-repr\'esentation $V_x$ associ\'ee \`a $x$
est une extension non triviale de $(\brig{}{,\Qp} \otimes K(x))(\delta_{2,x})$
par $(\brig{}{,\Qp} \otimes K(x))(\delta_{1,x})$. On note $$\mathcal{S}' \subset \mathcal{S}$$ l'ouvert admissible de 
$\mathcal{S}$ d\'efini par la condition $(\delta_{1,x}/\delta_{2,x})(p) \notin p^\ZZ$. L'application naturelle 
$\mathcal{S}' \to \mathcal{D}$ est alors une immersion ouverte, et d'apr\`es \cite[Prop. 5.2]{CTR}, pour tout 
$x \in \mathcal{S}'$ il existe un voisinage affino\"ide $\mathcal{B}_x$ de $x$
dans $\mathcal{S}'$ qui est une
$K(x)$-boule, ainsi qu'une 
repr\'esentation continue $\rho^{\mathcal{B}_x} : G_{\Qp} \to \GL_2(\OO(\mathcal{B}_x))$, dont l'\'evaluation en chaque 
$y\in \mathcal{B}_x$ est isomorphe \`a $V_y$. Comme $\mathcal{S}'$ est de type d\'enombrable, car $\mathcal{D}$ l'est, 
on peut extraire du recouvrement des $\mathcal{B}_x$ un recouvrement d\'enombrable $\{ \mathcal{B}_x \}_{x \in I}$ avec 
$I \simeq \ZZ_{\geq 0}$. On pose alors $$\mathcal{S}'_{\rm dec} = \coprod_{x \in I} \mathcal{B}_x.$$ Il s'agit d'une 
``d\'econnexion'' non canonique de $\mathcal{S}'$. On a par ailleurs une immersion ouverte surjective \'evidente 
$\mathcal{S}'_{\rm dec} \to \mathcal{S}'$ ainsi qu'une repr\'esentation galoisienne naturelle 
$\rho^{\mathcal{S}'_{\rm dec}} : G_{\Qp} \to \GL_2(\OO(\mathcal{S}'_{\rm dec}))$
obtenue \`a partir des $\rho^{\mathcal{B}_x}$ pour $x \in I$. Notons enfin que
via l'inclusion $\mathcal{S}' \subset \mathcal{D}$, l'op\'eration consistant \`a fixer le polyn\^ome de Sen et le d\'eterminant en $p$ est encore parfaitement transparente, et que les lieux obtenus sont de type d\'enombrable et \'equi-dimensionnels de dimension $4-3=1$.

Il nous faut maintenant comprendre $\mathcal{S}\backslash \mathcal{S}'$. Remarquons que lorsque le polyn\^ome de Sen d'une 
repr\'esentation trianguline $x \in \mathcal{S}$ est donn\'e, on conna\^{\i}t les deux caract\`eres 
${\delta_{1,x}}_{|\Zp^\times}$ et ${\delta_{2,x}}_{|\Zp^\times}$  \`a des caract\`eres d'ordre fini pr\`es. Si de plus on 
ne s'interesse qu'\`a des repr\'esentations dans $\mathcal{S}\backslash \mathcal{S}'$ dont le d\'eterminant en $p$ est 
aussi donn\'e, cela d\'etermine un nombre d\'enombrable de $\delta_{i,x}(p)$ possibles, et donc de paires de 
$\delta_{i,x}$ possibles. Pour chacune de ces paires, il y a en fait une et une seule repr\'esentation trianguline 
associ\'ee dans $\mathcal{S}\backslash \mathcal{S}'$, \`a moins que
$\delta_{1,x}/\delta_{2,x}$ ne soit de la forme $z \mapsto z^i|z|$ pour un entier $i\geq 1$, 
d'apr\`es \cite[Thm. 2.9]{CTR}.\footnote{Notons que l'autre cas a priori exceptionnel,
o\`u $\delta_{1,x}/\delta_{2,x}$ est de la forme $z \mapsto z^{-i}$ pour $i\geq 0$, ne se produit pas pour 
$x \in \mathcal{S}$, car il ne correspond pas \`a un $(\varphi,\Gamma_{\Qp})$-module \'etale si $i\neq 0$, et 
qu'il est r\'eductible si $i=0$.} Mais dans ce cas, les repr\'esentations possibles sont des torsions par un 
caract\`ere de repr\'esentations semi-stables, et sont donc du type (a). 

Terminons enfin par le type (a). Le lieu de de Rham de $\mathcal{X}(r)$, s'il est non vide, est un ferm\'e 
analytique (cf \cite{LB11}). En particulier, il est de type d\'enombrable. Sa dimension a \'et\'e calcul\'ee par 
Kisin dans \cite[Thm. 3.3.8 ]{MK} : elle est toujours $\leq 1$. Par torsion, on en d\'eduit que pour tout caract\`ere 
$\eta : G_{\Qp} \to F^\times$ ($F$ \'etant une extension finie de $E$), 
le lieu des $x \in \mathcal{X}(r) \times_E F$ tels que $r_x \otimes \eta$
soit de de Rham est encore un ferm\'e analytique de dimension $\leq 1$. 

Pour r\'ecapituler, nous avons d\'emontr\'e le r\'esultat suivant:

\begin{lemm} Il existe un $E$-espace analytique $\mathcal{Y}$, ainsi qu'une repr\'esentation continue $\rho : G_{\Qp} \rightarrow \OO(\mathcal{Y})^\times$, tels que:
\begin{itemize}
\item $\mathcal{Y}$ est de type d\'enombrable et de dimension $\leq 1$,
\item le polyn\^ome de Sen de $\rho$ est constant \'egal \`a $P_0$, et son d\'eterminant en $p$ est constant \'egal \`a $\lambda_0$,
\item pour tout point $x \in \mathcal{X}_0$ potentiellement triangulin, il existe $y \in \mathcal{Y}$, et des plongements de $K(x)$ 
et $K(y)$ dans  $\Qpbar$, tels que $\rho_y \otimes_{K(y)} \Qpbar \simeq r_x \otimes_{K(x)} \Qpbar$,
\item pour tout $y \in \mathcal{Y}$, la repr\'esentation $\rho_y$ est potentiellement trianguline.
\end{itemize}
\end{lemm}

En effet, on peut prendre pour $\mathcal{Y}$ la r\'eunion disjointe de : l'ensemble d\'enombrable des points de type (b), 
l'ensemble d\'enombrable des points de type (c) dans $\mathcal{S}\backslash \mathcal{S}'$
qui ne sont pas de type (a), l'espace 
$\mathcal{S}'_{\rm dec}$, et pour chaque $\eta : G_{\Qp} \to F^\times$ dans un ensemble d\'enombrable,
du lieu de de Rham de $\mathcal{X}(r) \times_E F$.

Le th\'eor\`eme \ref{gchu} est maintenant imm\'ediat : soit $\mathcal{Y}(r) \subset \mathcal{Y}$ l'ouvert ferm\'e sur 
lequel la repr\'esentation $\rho$ ci-dessus est r\'esiduellement constante et isomorphe \`a $r$. Par la propri\'et\'e 
universelle de $\mathcal{X}(r)$, et donc de $\mathcal{X}_0$, la repr\'esentation $\rho$ correspond \`a un morphisme 
analytique $f : \mathcal{Y}(r) \to \mathcal{X}_0$. De plus, $f(\mathcal{Y}(r))$ est exactement l'ensemble des points 
potentiellement triangulins de $\mathcal{X}_0$. Il vient que
$f(\mathcal{Y}(r))$ est fine dans $\mathcal{X}_0$ car chaque composante irr\'eductible de $\mathcal{X}_0$ est 
de dimension $\geq 2$.  

\bibliographystyle{smfalpha}
\bibliography{potrig}

\providecommand{\bysame}{\leavevmode ---\ }
\providecommand{\og}{``}
\providecommand{\fg}{''}
\providecommand{\smfandname}{et}
\providecommand{\smfedsname}{\'eds.}
\providecommand{\smfedname}{\'ed.}
\providecommand{\smfmastersthesisname}{M\'emoire}
\providecommand{\smfphdthesisname}{Th\`ese}
\begin{thebibliography}{Maz89}

\bibitem[BC08]{LB11}
{\scshape L.~Berger {\normalfont \smfandname} P.~Colmez} -- {\og Familles de
  repr\'esentations de de {R}ham et monodromie $p$-adique\fg},
  \emph{Ast\'erisque} (2008), no.~319, p.~303--337.

\bibitem[BC09]{BCH}
{\scshape J.~Bella{\"{\i}}che {\normalfont \smfandname} G.~Chenevier} -- {\og
  Families of {G}alois representations and {S}elmer groups\fg}, Ast\'erisque
  324, to appear, 2009.

\bibitem[Ber08]{LB8}
{\scshape L.~Berger} -- {\og Construction de {$(\varphi,\Gamma)$}-modules:
  repr{\'e}sentations {$p$}-adiques et {$B$}-paires\fg}, \emph{Algebra Number
  Theory} \textbf{2} (2008), no.~1, p.~91--120.

\bibitem[Col08]{CTR}
{\scshape P.~Colmez} -- {\og Repr{\'e}sentations triangulines de dimension
  $2$\fg}, \emph{Ast\'erisque} (2008), no.~319, p.~213--258.

\bibitem[Fon94]{FPP}
{\scshape J.-M. Fontaine} -- {\og Le corps des p\'eriodes {$p$}-adiques\fg},
  \emph{Ast\'erisque} (1994), no.~223, p.~59--111, With an appendix by Pierre
  Colmez, P{\'e}riodes $p$-adiques (Bures-sur-Yvette, 1988).

\bibitem[Fon04]{FIHP}
\bysame , {\og Arithm\'etique des repr\'esentations galoisiennes
  {$p$}-adiques\fg}, \emph{Ast\'erisque} (2004), no.~295, p.~xi, 1--115,
  Cohomologies $p$-adiques et applications arithm{\'e}tiques. III.

\bibitem[Ked04]{KLMT}
{\scshape K.~S. Kedlaya} -- {\og A {$p$}-adic local monodromy theorem\fg},
  \emph{Ann. of Math. (2)} \textbf{160} (2004), no.~1, p.~93--184.

\bibitem[Kie67]{K67}
{\scshape R.~Kiehl} -- {\og Die de {R}ham {K}ohomologie algebraischer
  {M}annigfaltigkeiten \"uber einem bewerteten {K}\"orper\fg}, \emph{Inst.
  Hautes \'Etudes Sci. Publ. Math.} (1967), no.~33, p.~5--20.

\bibitem[Kis08]{MK}
{\scshape M.~Kisin} -- {\og Potentially semi-stable deformation rings\fg},
  \emph{J. Amer. Math. Soc.} \textbf{21} (2008), no.~2, p.~513--546.

\bibitem[Maz89]{MDG}
{\scshape B.~Mazur} -- {\og Deforming {G}alois representations\fg}, Galois
  groups over {${\bf Q}$} ({B}erkeley, {CA}, 1987), Math. Sci. Res. Inst.
  Publ., vol.~16, Springer, New York, 1989, p.~385--437.

\bibitem[Nak09]{KN}
{\scshape K.~Nakamura} -- {\og Classification of two-dimensional split
  trianguline representations of {$p$}-adic fields\fg}, \emph{Compos. Math.}
  \textbf{145} (2009), no.~4, p.~865--914.

\bibitem[Sen81]{SN80}
{\scshape S.~Sen} -- {\og Continuous cohomology and {$p$}-adic {G}alois
  representations\fg}, \emph{Invent. Math.} \textbf{62} (1980/81), no.~1,
  p.~89--116.

\end{thebibliography}
\end{document}